\documentclass{article}
\usepackage[utf8]{inputenc}
\usepackage[letterpaper,margin=1in]{geometry}
\usepackage{amsmath, amssymb, amsthm, mathtools}
\usepackage{tikz}
\usetikzlibrary{arrows.meta}
\usepackage{enumitem}
\usepackage{lineno}

\newcommand{\ds}{\displaystyle}
\newtheorem{theorem}{Theorem}[section]
\newtheorem{proposition}[theorem]{Proposition} 
\newtheorem{lemma}[theorem]{Lemma}
\newtheorem{corollary}[theorem]{Corollary}
\theoremstyle{definition}\newtheorem{remark}[theorem]{Remark}
\theoremstyle{definition}\newtheorem{example}[theorem]{Example}

\DeclareMathOperator{\Child}{Chi}
\DeclareMathOperator{\parent}{par}
\DeclareMathOperator{\Gen}{Gen}
\DeclareMathOperator{\Span}{span}
\DeclareMathOperator{\sgn}{sgn}
\newcommand{\Z}{\mathbb{Z}}
\newcommand{\R}{\mathbb{R}}
\newcommand{\N}{\mathbb{N}}
\newcommand{\C}{\mathbb{C}}
\newcommand{\K}{\mathbb{K}}
\begin{document}

\title{Chain recurrent shifts on trees}

\author{
  Andrew Mortensen \\
  \small Department of Mathematics, Statistics, and Computer Science \\
  \small St.\ Olaf College, Northfield, MN 55057, USA \\
  \small \texttt{morten3@stolaf.edu}
  \and
  David Walmsley\thanks{Corresponding author.} \\
  \small Department of Mathematics, Statistics, and Computer Science \\
  \small St.\ Olaf College, Northfield, MN 55057, USA \\
  \small \texttt{walmsl1@stolaf.edu}
}

\date{\today}
\maketitle

\begin{abstract}
    We characterize when a weighted backward shift is chain recurrent on the $\ell^p$ ($1\leq p<\infty$) and $c_0$ spaces of a directed tree. The characterization is given in terms of two divergence conditions on the weights: a forward condition on the descendants of each vertex and, in the unrooted case, a backward condition on the descendants of each ancestor. The conditions reduce, in the case of symmetric weighted shifts on symmetric trees, to the classical characterizations of chain recurrence on the sequence spaces $\ell^p(\N)$, $\ell^p(\Z)$, $c_0(\N)$, and $c_0(\Z)$.
\end{abstract}



\section{Introduction}\label{Introduction}
For a map $f:X\to X$ on a metric space $X$ with metric $d$, a \emph{$\delta$-pseudotrajectory} of $f$ is a finite or infinite sequence $(x_{j})_{i<j<k}$ in $X$, where $-\infty\le i<k\le\infty$ and $k-i\ge3$, such that 
\[d(f(x_{j}),x_{j+1})\le\delta \text{ \hskip .1in for all \hskip .1in} i<j<k-1. \]
A finite $\delta$-pseudotrajectory of the form $(x_{j})_{j=0}^{k}$ is also called a \emph{$\delta$-chain} for $f$ from $x_{0}$ to $x_{k}$, with $k$ its length. The map $f$ is called \emph{chain recurrent} (resp. \emph{chain transitive}) if for every $x\in X$ (resp. $x, y\in X$) and every $\delta>0$, there is a $\delta$-chain for $f$ from $x$ to itself (resp. from $x$ to $y$). 
Furthermore, $f$ is called \emph{chain mixing} if for every $x, y\in X$ and every $\delta>0$, there exists $k_{0}\in\mathbb{N}$ such that for every $k\ge k_{0}$, there is a $\delta$-chain for $f$ from $x$ to $y$ with length $k$. In general, the following implications are immediate consequences of the definitions involved:
    \begin{center}
        chain mixing $\Rightarrow$ chain transitive $\Rightarrow$ chain recurrent.
    \end{center}
When $X$ is a topological vector space and $T:X\to X$ is linear (not necessarily continuous), all three notions are in fact equivalent, as was observed in \cite[Proposition 19]{(1)Alves}. 
    
For a map $f:X\to X$, let $CR(f)$ be its set of chain recurrent points, specifically $CR(f)=\{x\in X: \forall \, \delta>0, \exists \, \delta$-chain for $f$ from $x$ to itself$\}$. Then $f$ is chain recurrent if and only if $CR(f)=X$. Moreover, we say that $f$ has the \emph{positive shadowing property} if for every $\epsilon>0$, there exists $\delta>0$ such that every $\delta$-pseudotrajectory $(x_{j})_{j\in\mathbb{N}_{0}}$ of $f$ is $\epsilon$-shadowed by a real trajectory of $f$, that is, there exists $x\in X$ such that $d(x_{j},f^{j}(x))<\epsilon$ for all $j\in\mathbb{N}_{0}$. If $f$ is bijective, then the shadowing property is defined by replacing the set $\mathbb{N}_{0}$ by the set $\mathbb{Z}$ in the above definition. 

The concepts of pseudotrajectories, shadowing, and chain recurrence trace their roots back to the foundational research of Conley \cite{(20)Conley}, Sina\u{\i} \cite{(43)Sinai}, and Bowen \cite{(15)Bowen} in the early 1970s. These ideas are essential to the qualitative analysis of differential equations and, more generally, dynamical systems. For a detailed overview of these concepts and their practical uses, we refer the reader to the books \cite{(4)Aoki,(22)Devaney,(29)Katok,(39)Palmer,(40)Pilyugin,(42)Shub}. 

Only recently have shadowing and chain recurrence been studied within linear dynamics \cite{(1)Alves,(3)Antunes,(9)Bernardes,(11)Bernardes,(16)Cirilo}. For those interested in learning more, we recommend the excellent work of Bernardes and Peris in \cite{BernardesPeris} as a rich source of information. A recurring theme has been that the linear setting can differ markedly from the classical compact one. For example, the first examples of non-hyperbolic operators with the shadowing property were exhibited in \cite{(9)Bernardes}. Along other lines, \cite{LopezPapachainrecontrees} gave an example of a continuous linear operator $T:X\to X$ on a Banach space $X$ for which the restriction of $T$ to its set of chain recurrent vectors $CR(T)$ is not chain recurrent, in contrast to the classical compact setting where \cite[Theorem 3.1.6]{(4)Aoki} guarantees that $f|_{CR(f)}$ is always chain recurrent for a homeomorphism $f:K\to K$ on a compact metric space $K$. 

The example of \cite{LopezPapachainrecontrees} was a weighted backward shift on a directed tree, and that setting is the primary focus of our present work. While \cite{LopezPapachainrecontrees} uses such shifts as a tool to construct a specific counterexample answering a question of Bernardes and Peris \cite{BernardesPeris}, our aim here is the systematic characterization of when chain recurrence occurs within this class. The ``classical Banach sequence spaces'' are $\ell^p(\N), 1\leq p<\infty$, and $c_0(\N)$, with their counterparts $\ell^p(\Z), 1\leq p<\infty$, and $c_0(\Z)$. Viewing $\N$ as a rooted directed tree and $\Z$ as an unrooted directed tree, it is natural to expand the family of weighted shifts by replacing $\N$ and $\Z$ with other directed trees. The dynamical study of such operators was initiated by Mart\'inez-Avenda\~no \cite{Martinez}, and the program was carried forward by Grosse-Erdmann and Papathanasiou \cite{GrossePapahypercyclic}, who obtained a complete characterization of the hypercyclicity, weak mixing, and mixing properties of weighted backward shifts on the classical Banach sequence spaces of directed trees. 

The program initiated in \cite{GrossePapahypercyclic} has since been extended in several directions, each characterizing a different dynamical property of these shifts in terms of the weights. Grosse-Erdmann and Papathanasiou themselves \cite{GrossePapachaotic} characterized the chaotic shifts, a class strictly contained in the hypercyclic shifts. More recently, Abakumov and Abbar \cite{AbakumovAbbar} characterized $\mathcal{F}$-transitivity (and the associated $\mathcal{F}$-recurrence) for Furstenberg families $\mathcal{F}$, in the framework of B\`es, Menet, Peris, and Puig \cite{BMPP}; as a consequence \cite[Corollary 5.2]{AbakumovAbbar}, they showed that on unrooted trees ordinary topological recurrence coincides with hypercyclicity. Our goal in this paper is to add a complementary entry to this program: we characterize, in terms of the weights, precisely when a weighted backward shift on a classical Banach sequence space on a directed tree is chain recurrent. Since chain recurrence is, in general, weaker than hypercyclicity in linear dynamics \cite[Proposition 19]{(1)Alves}, the class of operators we characterize sits one level below the previously studied ones, properly containing the hypercyclic shifts of \cite{GrossePapahypercyclic}, the chaotic shifts of \cite{GrossePapachaotic}, and the $\mathcal{F}$-recurrent shifts on unrooted trees of \cite{AbakumovAbbar}. 

The paper is organized as follows. Section \ref{General background and definitions} collects the necessary background on directed trees, sequence spaces, and weighted shifts, and includes two reductions (Propositions \ref{leafless prop} and \ref{forward shift prop}) showing that the analysis can be restricted to weighted backward shifts on leafless trees. Section \ref{Characterization section} contains our main characterizations (Theorems \ref{unrooted theorem} and \ref{rooted theorem}) followed by worked examples of Rolewicz operators and symmetric weighted shifts on symmetric trees, the latter of which allows us to recover classical characterizations of chain recurrence on the sequence spaces $\ell^p(\N)$, $\ell^p(\Z)$, $c_0(\N)$, and $c_0(\Z)$.

\section{General background and definitions}\label{General background and definitions}

\subsection{Directed trees}
As in \cite{Jablonski,GrossePapahypercyclic,GrossePapachaotic,LopezPapachainrecontrees}, a \textit{directed tree} $(V,E)$ is a connected directed graph with a countable vertex set $V$ and a set of directed edges $E\subset V\times V \setminus \{(v,v): v\in V\}$ such that:
\begin{itemize}
    \item $(V,E)$ has no cycles;
    \item each vertex $v\in V$ has at most one \textit{parent}, that is, a unique vertex denoted by $\parent(v)\in V$ such that $(\parent(v),v)\in E$;
    \item there is at most one vertex with no parent, called the \textit{root} and denoted by $\mathsf{root}$ when it exists. 
\end{itemize}
Any vertex whose parent is $v$ is called a \textit{child} of $v$, and the set of children of $v$ is denoted by $\Child(v)$. Inductively, we define $\Child^0 (v)=\{v\}$, and for $n\geq 1$, $\Child^n (v)$ is the union of the sets of children of each vertex in $\Child^{n-1}(v)$. A vertex with no children is called a \emph{leaf}. If the tree has a root, we call $\Child^n(\mathsf{root})$ the \emph{$n$th generation} of the tree, $n\in \N_0$, and denote the $n$th generation by $\Gen_n$. It is easy to see that the generations of a rooted tree partition the vertex set. Following \cite{GrossePapahypercyclic}, for an unrooted tree, we fix a vertex $v_0\in V$ and define, for $n\in \Z$, the \emph{$n$th generation} with respect to $v_0$ as 
\begin{align*}
    \Gen_n = \Gen_n(v_0) = \{ v \in V : \exists m \geq \max(-n,0) \text{ such that } \parent^{n+m}(v) = \parent^m(v_0) \}.
\end{align*}
It is also not difficult to see that the generations of an unrooted tree partition the vertex set; we prove this in Lemma \ref{partition lemma} below. Figure \ref{fig:grouped_generations_tree} illustrates the generations of a typical unrooted tree. 

\begin{figure}[htbp]
    \centering
    \begin{tikzpicture}[
        vertex/.style={circle, fill=black, inner sep=1.2pt},
        edge/.style={-Latex, semithick},
        dashed edge/.style={-Latex, semithick, dashed}
    ]
        
        \draw[dashed edge] (-1.5, 0) -- (0, 0);
        
        
        \draw[thick, dotted, gray] (0, 0) ellipse [x radius=0.5, y radius=0.7];
        \node[above, font=\bfseries] at (0, 0.7) {$\Gen_{-2}$};
        
        \draw[thick, dotted, gray] (2, 0) ellipse [x radius=0.7, y radius=2.6];
        \node[above, font=\bfseries] at (2, 2.6) {$\Gen_{-1}$};
        
        \draw[thick, dotted, gray] (4, 0) ellipse [x radius=0.7, y radius=3.2];
        \node[above, font=\bfseries] at (4, 3.2) {$\Gen_0$};
        
        \draw[thick, dotted, gray] (6, 0.1) ellipse [x radius=0.7, y radius=3.7];
        \node[above, font=\bfseries] at (6, 3.8) {$\Gen_1$};

        
        \node[vertex] (root) at (0, 0) {};
        
        \node[vertex] (t1) at (2, 2) {};
        \node[vertex] (m1) at (2, 0) {};
        \node[vertex] (b1) at (2, -2) {};
        
        \node[vertex] (t2_1) at (4, 2.7) {};
        
        \node[vertex, label=above:$v_0$] (t2_2) at (4, 1.3) {}; 
        
        \node[vertex] (m2_1) at (4, 0) {};
        
        \node[vertex] (b2_1) at (4, -1.3) {};
        \node[vertex] (b2_2) at (4, -2.7) {};
        
        \node[vertex] (t3_1) at (6, 3.2) {};
        \node[vertex] (t3_2) at (6, 2.2) {};
        
        \node[vertex] (t3_3) at (6, 1.6) {};
        \node[vertex] (t3_4) at (6, 1.0) {};
        
        \node[vertex] (m3_1) at (6, 0) {};
        
        \node[vertex] (b3_1) at (6, -1.3) {};
        
        \node[vertex] (b3_2) at (6, -2.4) {};
        \node[vertex] (b3_3) at (6, -3.0) {};

        
        \draw[edge] (root) -- (t1);
        \draw[edge] (root) -- (m1);
        \draw[edge] (root) -- (b1);
        
        \draw[edge] (t1) -- (t2_1);
        \draw[edge] (t1) -- (t2_2);
        
        \draw[edge] (m1) -- (m2_1);
        
        \draw[edge] (b1) -- (b2_1);
        \draw[edge] (b1) -- (b2_2);
        
        \draw[edge] (t2_1) -- (t3_1);
        \draw[edge] (t2_1) -- (t3_2);
        
        \draw[edge] (t2_2) -- (t3_3); 
        \draw[edge] (t2_2) -- (t3_4); 
        
        \draw[edge] (m2_1) -- (m3_1);
        
        \draw[edge] (b2_1) -- (b3_1);
        \draw[edge] (b2_2) -- (b3_2);
        \draw[edge] (b2_2) -- (b3_3);
        
        
        \draw[dashed edge] (t3_1) -- (7.5, 3.2);
        \draw[dashed edge] (t3_2) -- (7.5, 2.2);
        
        \draw[dashed edge] (t3_3) -- (7.5, 1.6);
        \draw[dashed edge] (t3_4) -- (7.5, 1.0);
        
        \draw[dashed edge] (m3_1) -- (7.5, 0);
        
        \draw[dashed edge] (b3_1) -- (7.5, -1.3);
        
        \draw[dashed edge] (b3_2) -- (7.5, -2.4);
        \draw[dashed edge] (b3_3) -- (7.5, -3.0);
        
    \end{tikzpicture}
    \caption{An unrooted directed tree with vertex $v_0$ and its generations.}
    \label{fig:grouped_generations_tree}
\end{figure}

\begin{lemma}\label{partition lemma}
    Let $(V,E)$ be an unrooted directed leafless tree, and fix $v_0\in V$ so as to enumerate the generations of the tree. If $k<n$, then 
    \[\bigcup_{v\in \Gen_k} \Child^{n-k}(v)=\Gen_n \text{ and } \Gen_k = \bigcup_{w\in \Gen_n} \{\parent^{n-k}(w)\}.\]
    Consequently, $|\Gen_k|\leq |\Gen_n|$ whenever $k<n$.
\end{lemma}

\begin{proof}
    To show $\bigcup_{v\in \Gen_k} \Child^{n-k}(v) \subseteq \Gen_n$, let $u\in \bigcup_{v\in \Gen_k} \Child^{n-k}(v).$ Then there exists $v\in \Gen_k$ for which $v=\parent^{n-k}(u)$. Since $v\in \Gen_k$, there exists an integer $m\geq \max(-k,0)$ for which $\parent^{m+k}(v)=\parent^m(v_0)$. Substituting $v=\parent^{n-k}(u)$ into the previous expression and simplifying yields $\parent^{m+n}(u)=\parent^m(v_0)$, and since $m\geq \max(-k,0)\geq \max(-n,0)$, we have that $u\in \Gen_n$.

    To show the reverse containment $\Gen_n\subseteq \bigcup_{v\in \Gen_k} \Child^{n-k}(v)$, let $u\in \Gen_n$, which implies the existence of an integer $m\geq \max(-n,0)$ such that $\parent^{m+n}(u)=\parent^m(v_0)$. Let $v=\parent^{n-k}(u)$ and $l=m+\max(-k,0)$. Then $l\geq \max(-k,0)$, and we compute that
    \begin{align*}
        \parent^{l+k}(v)=\parent^{l+n}(u)=\parent^{l-m}(\parent^{m+n}(u))=\parent^{l-m}(\parent^m(v_0))=\parent^l(v_0).
    \end{align*}
    Hence $v\in \Gen_k$, and $v=\parent^{n-k}(u)$ implies $u\in \Child^{n-k}(v)$, which establishes the reverse containment and shows $\bigcup_{v\in \Gen_k} \Child^{n-k}(v)=\Gen_n$.  Since the tree is leafless, each vertex has at least one child, and we can deduce from the set equality $\bigcup_{v\in \Gen_k} \Child^{n-k}(v)=\Gen_n$ that $|\Gen_k|\leq |\Gen_n|$ whenever $k<n$.
    
    The set equality $\Gen_k = \bigcup_{w\in \Gen_n} \{\parent^{n-k}(w)\}$ can be proved with a similar argument.
\end{proof}

As in Figure \ref{fig:grouped_generations_tree}, we  imagine our trees to be directed from left to right, a convention justified by Lemma \ref{partition lemma}. We say an unrooted tree has a \emph{free left end} if there is some $n\in\Z$ for which $\Gen_k$ is a singleton for all $k\leq n$. There is a simple criterion to check if an unrooted tree has a free left end. 

\begin{lemma}\label{free left end lemma}
    Let $(V,E)$ be an unrooted directed leafless tree, and fix $v_0\in V$ so as to enumerate the generations of the tree. Then $V$ has a free left end if and only if there exists $n\in \Z$ for which $|\Gen_n|<\infty$.
\end{lemma}

\begin{proof}
    The ($\Rightarrow$) direction is obvious. For the ($\Leftarrow$) direction, suppose $|\Gen_n|<\infty$ for some $n\in \Z$. If $|\Gen_n|=1$, then for every $k<n$, Lemma \ref{partition lemma} implies $|\Gen_k|\leq |\Gen_n|=1$, which shows $V$ has a free left end. So suppose $1<|\Gen_n|<\infty$. It suffices to find an integer $k<n$ with $|\Gen_k|<|\Gen_n|$. 

    Since $|\Gen_n|>1$, there exist vertices $u,v\in \Gen_n$ with $u\neq v$ and integers $l\geq \max(-n,0)$ and $m\geq \max(-n,0)$ such that 
    \begin{align*}
        \parent^{m+n}(v)=\parent^m(v_0) \text{ and } \parent^{l+n}(u)=\parent^l(v_0).
    \end{align*}
    Without loss of generality, we may assume $m\geq l$. Then $\parent^{m+n}(u)=\parent^m(v_0)$ as well, since
    \begin{align*}
        \parent^{m+n}(u)=\parent^{m-l}(\parent^{l+n}(u))=\parent^{m-l}(\parent^{l}(v_0)) = \parent^m(v_0).
    \end{align*}
    It is impossible for $m$ to equal $-n$, since that would imply, by the above equations, that $u=\parent^m(v_0)=v$. Then since $m\geq \max(-n,0)$, we can deduce that $m>-n$.
    
    Since $\parent^m(v_0)\in \Gen_{-m}$, we have
    \begin{align*}
        \{u,v\}\subseteq \Child^{n+m}(\parent^m(v_0))\subseteq \Gen_n,
    \end{align*}
    where the last set containment follows from Lemma \ref{partition lemma}. Thus the mapping $g:\Gen_n\to \Gen_{-m}$ given by $g(w)=\parent^{n+m}(w)$ is not an injective mapping, since $u,v\in \Gen_n$ and $g(u)=g(v)$. But the map $g$ is surjective by Lemma \ref{partition lemma}, and the existence of a non-injective surjection between the finite sets $\Gen_n$ and $\Gen_{-m}$ shows $|\Gen_{-m}|<|\Gen_n|$. Iterating this strict-reduction argument finitely many times produces a generation of cardinality one, so $V$ has a free left end.
\end{proof}

\subsection{Sequence spaces and weighted shifts on directed trees}
Let $\K$ denote either $\R$ or $\C$ and let $V$ be an arbitrary finite or countable set. We denote by $\K^V$ the space of all real or complex sequences $f = (f(v))_{v \in V}$ over $V$, and endow $\K^V$ with the product topology. The canonical unit sequences
are denoted by $e_v = \chi_{\{v\}}$, $v \in V$; that is,
\begin{align*}
    e_v(u):=\begin{cases} 
        1 & \text{if } u=v, \\
        0 & \text{if } u\neq v. \\
    \end{cases}
\end{align*} 
A subspace of $\K^V$ is called a \emph{Banach sequence space over $V$} if it is endowed with a Banach space topology for which the canonical embedding into $\K^V$ is continuous. 

Our focus is on the following Banach sequence spaces over $V$:
\begin{align*}
    \ell^p(V) =\left\{ f \in \K^V : \|f\|_p \coloneqq \sum_{v \in V} |f(v)|^p < \infty \right\}, \, 1 \leq p < \infty,
\end{align*}
with the typical modification for $\ell^\infty(V)$, and
\begin{align*}
    c_0(V) = \left\{f \in \mathbb{K}^V : \forall \varepsilon > 0,\ \exists F \subset V \text{ finite}, \, \forall v \in V \setminus F,\ |f(v)| < \varepsilon
\right\},
\end{align*}
which is endowed with the canonical norm of $\ell^\infty(V)$. We refer to elements of these spaces as \emph{functions} rather than \emph{sequences}, to avoid confusion when discussing sequences of such elements. For a Banach sequence space $X$ above, we denote by $0_X$ the additive identity in $X$. 

Gathering several facts from the literature, we present next a lemma which is crucial for our main result later. 
\begin{lemma}\label{recurrence lemma}
    Let $V$ be a finite or countable set, let $X$ be $c_0$ or one of the $\ell^p(V)$ spaces, and let $\|\cdot\|$ denote the norm of the space $X$. The following statements hold for every linear operator $T:X\to X$.
    \begin{enumerate}
        \item[(a)] Let $(f_l)_{l=0}^n$ be a finite sequence in $X$. Then $(f_l)_{l=0}^n$ is a $\delta$-chain for $T$ if and only if the finite sequence $g_l\coloneqq f_l-T(f_{l-1})$, $1\leq l \leq n$, fulfills that $f_n = T^n(f_0) + \sum_{l=1}^n T^{n-l}(g_l)$ and $\|g_l\|<\delta$ for every $1\leq l \leq n$.
        \item[(b)] $T$ is chain recurrent if and only if for every canonical basis vector $e_v\in X$ and every $\delta>0$, there is a $\delta$-chain for $T$ from $e_v$ to $0_X$ and a $\delta$-chain for $T$ from $0_X$ to $e_v$.
    \end{enumerate}
\end{lemma}
    \begin{proof}
        The ($\Rightarrow$) direction of statement (a) is an immediate consequence of \cite[Lemma 2.2(a)]{LopezPapachainrecontrees} and its proof. For the ($\Leftarrow$) direction, $(f_l)_{l=0}^n$ is a $\delta$-chain for $T$ because for every $1\leq l \leq n$, the relationship $g_l=f_l-T(f_{l-1})$ implies $\|f_l-T(f_{l-1})\|=\|g_l\|<\delta.$

        The ($\Rightarrow$) direction of statement (b) follows immediately from the fact that chain recurrence and chain transitivity are equivalent notions for the linear operator $T$. For the ($\Leftarrow$) direction, since $T(0_X)=0_X$, we can concatenate a $\delta$-chain from $e_v$ to $0_X$ with another $\delta$-chain from $0_X$ to $e_v$ to create a $\delta$-chain from $e_v$ to $e_v$. Hence every basis vector $e_v$ is a chain recurrent vector for $T$. Since the set of chain recurrent vectors for $T$ is a closed subspace of $X$ by \cite[Lemma 8 and Proposition 26]{BernardesPeris}, the set of chain recurrent vectors must contain $\overline{\Span\{e_v: v\in V\}}=X$.
    \end{proof}

\begin{remark}
    The cited results \cite[Lemma 2.2(a)]{LopezPapachainrecontrees} and \cite[Lemma 8 and Proposition 26]{BernardesPeris} are proved in more general settings than just the Banach sequence spaces considered here. Consequently, Lemma \ref{recurrence lemma} carries over to those settings verbatim. For example, both statements hold for any Fr\'echet sequence space $X$ over $V$ in which the canonical vectors $(e_v)_{v\in V}$ form an unconditional basis: statement (a) follows from \cite[Lemma 2.2(a)]{LopezPapachainrecontrees} applied to the linear map $T:X\to X$ on the Fr\'echet space $X$, while the closedness of $CR(T)$ as a subspace, on which statement (b) relies, is part of \cite[Lemma 8 and Proposition 26]{BernardesPeris} in this generality. We do not need this level of generality for our purposes, so we content ourselves with the statement above.
\end{remark} 

\subsection{Weighted backward and forward shifts}
Suppose $(V,E)$ is a directed tree and we have chosen a sequence $\lambda\coloneqq (\lambda_v)_{v\in V} \in \K^V$ of non-zero scalars called \textit{weights}. Then the associated \textit{weighted backward shift} $B_\lambda$ on $V$ is formally defined by
\begin{align*}
    [B_{\lambda}(f)](v) \coloneqq \sum_{u \in \Child(v)} \lambda_u f(u), \text{ for each } f=(f(v))_{v\in V}\in \K^V \text{ and each }  v \in V,
\end{align*}
while the associated \textit{weighted forward shift} $S_\lambda$ on $V$ is defined by
\begin{align*}
    [S_\lambda(f)](v)=
    \begin{cases}
        \lambda_v f(\parent(v)) & \text{ if } \parent(v) \text{ exists}\\
        0 & \text{ if } \parent(v) \text{ does not exist}.
    \end{cases}
\end{align*}
As a consequence of the closed graph theorem, $B_\lambda$ and $S_\lambda$ are continuous on a Banach sequence space $X$ once they are well-defined on $X$, meaning they map $X$ into $X$. Characterizing conditions for the continuity of $B_\lambda$ and $S_\lambda$ on $\ell^p(V)$, $1\leq p<\infty$, and on $c_0(V)$ can be found in \cite[Propositions 2.2 \& 2.3]{GrossePapahypercyclic}. 

As in \cite{GrossePapahypercyclic}, for simplicity in discussions of repeated shifts $B_{\lambda}^n$, for $v\in V$ and $u\in \Child^n(v)$ with $n\geq 1$, we define $\lambda(v\to u)\coloneqq\lambda_{\parent^{n-1}(u)}\dots\lambda_{\parent (u)}\lambda_u$. Thus $\lambda(v\to u)$ is the product of weights $\lambda_w$ for each vertex $w$ along the unique path from $v$ to $u$. Repeated applications of $B_\lambda$ yield that for any $n\in \N$, the value of $B_\lambda^n(f)$ at vertex $v$ is given by
\begin{align}\label{B^n formula}
    [B_{\lambda}^n(f)](v) = \sum_{u \in \Child^n(v)} \lambda(v\to u) f(u), \text{ for each } f=(f(v))_{v\in V}\in \K^V \text{ and each }  v \in V.
\end{align}
When $n=0$, we interpret $B_\lambda^n$ as the identity operator. Hence by defining $\lambda(v\to v)\coloneqq 1$ for each $v\in V$ and recalling $\Child^0(v)=\{v\}$, the above formula holds when $n=0$, as well.

The iterated formula \eqref{B^n formula} reveals an immediate obstruction at leaves: every iterate $B_\lambda^n(f)$, $n\geq 1$, vanishes at any leaf $v$, since $\Child^n(v)=\emptyset$. This obstruction is the basis for ruling out the possibility of chain recurrence whenever the tree has a leaf. 

\begin{proposition}\label{leafless prop}
    Let $(V,E)$ be a directed tree with at least one leaf, and let $X$ be one of the spaces $\ell^p(V)$ ($1\leq p<\infty$) or $c_0(V)$. Then no weighted backward shift $B_\lambda$ on $X$ is chain recurrent.
\end{proposition}

\begin{proof}
    Let $v$ be a leaf, so $\Child(v)=\emptyset$ and hence $[B_\lambda(f)](v)=0$ for every $f\in X$. We show there is no $1$-chain from $0_X$ to $e_v$, which by Lemma \ref{recurrence lemma}(b) implies $B_\lambda$ is not chain recurrent. Suppose, toward a contradiction, that $(f_l)_{l=0}^n$ is such a $1$-chain. Then $\|f_n-B_\lambda(f_{n-1})\|<1$, i.e., $\|e_v-B_\lambda(f_{n-1})\|<1$. Evaluating $e_v-B_\lambda(f_{n-1})$ at $v$ gives
    \[
    1=|e_v(v)-0|=|e_v(v)-[B_\lambda(f_{n-1})](v)|\leq \|e_v-B_\lambda(f_{n-1})\|<1,
    \]
    a contradiction.
\end{proof}

We now turn to forward shifts and show that such an operator can be chain recurrent only on the trees $V=\Z$ and $V=-\N$, on which it is isometrically conjugate to a weighted backward shift on $\Z$ or $\N$, respectively.

\begin{proposition}\label{forward shift prop}
    Let $(V,E)$ be a directed tree, and let $X$ be one of the spaces $\ell^p(V)$ ($1\leq p<\infty$) or $c_0(V)$. If either
    \begin{enumerate}[label=\upshape(\roman*)]
        \item $V$ has a root, or
        \item $V$ contains a vertex with at least two children,
    \end{enumerate}
    then no weighted forward shift $S_\lambda$ on $X$ is chain recurrent. Consequently, $S_\lambda$ can be chain recurrent only when $V=\Z$ or $V=-\N$, in which case $S_\lambda$ is isometrically conjugate to a weighted backward shift on $\Z$ or $\N$, respectively, whose chain recurrence is characterized in \cite[Proposition 20]{(1)Alves} and \cite[Theorem 14]{BernardesPeris}.
\end{proposition}
\begin{proof}
    For case (i), suppose $V$ has a root. Since $\parent(\mathsf{root})$ does not exist, $[S_\lambda(f)](\mathsf{root})=0$ for every $f\in X$. We show there is no $1$-chain from $0_X$ to $e_{\mathsf{root}}$, which by Lemma \ref{recurrence lemma}(b) implies $S_\lambda$ is not chain recurrent. Suppose, toward a contradiction, that $(f_l)_{l=0}^n$ is such a $1$-chain. Then $\|e_{\mathsf{root}}-S_\lambda(f_{n-1})\|<1$. Since $|h(\mathsf{root})|\leq \|h\|$ for every $h\in X$,
    \[
    1=|e_{\mathsf{root}}(\mathsf{root})-[S_\lambda(f_{n-1})](\mathsf{root})|\leq \|e_{\mathsf{root}}-S_\lambda(f_{n-1})\|<1,
    \]
    a contradiction.

    For case (ii), suppose $v\in V$ has two distinct children $v_1$ and $v_2$. Set
    \[
    \delta\coloneqq \frac{|\lambda_{v_2}|}{|\lambda_{v_1}|+|\lambda_{v_2}|}>0.
    \]
    We show there is no $\delta$-chain from $0_X$ to $e_{v_1}$, which by Lemma \ref{recurrence lemma}(b) implies $S_\lambda$ is not chain recurrent. Suppose, toward a contradiction, that $(f_l)_{l=0}^n$ is such a $\delta$-chain, and set $t\coloneqq f_{n-1}(v)$. Since $\parent(v_k)=v$, the definition of $S_\lambda$ gives $[S_\lambda(f_{n-1})](v_k)=\lambda_{v_k}t$ for $k=1,2$, so
    \[
    (e_{v_1}-S_\lambda(f_{n-1}))(v_1)=1-\lambda_{v_1}t, \qquad (e_{v_1}-S_\lambda(f_{n-1}))(v_2)=-\lambda_{v_2}t.
    \]
    Using the fact that $\|e_{v_1}-S_\lambda(f_{n-1})\|<\delta$, along with the fact that for every $h\in X$,  $|h(v_k)|\leq \|h\|$ for $k=1,2$, we obtain from the previous computations that $|1-\lambda_{v_1}t|<\delta$ and $|\lambda_{v_2}t|<\delta$. The latter gives $|\lambda_{v_1}t|<\delta|\lambda_{v_1}|/|\lambda_{v_2}|$, so
    \[
    1=|(1-\lambda_{v_1}t)+\lambda_{v_1}t|\leq |1-\lambda_{v_1}t|+|\lambda_{v_1}t|<\delta+\delta\frac{|\lambda_{v_1}|}{|\lambda_{v_2}|}=\delta\cdot\frac{|\lambda_{v_1}|+|\lambda_{v_2}|}{|\lambda_{v_2}|}=1,
    \]
    a contradiction.

    For the final assertion: the only trees not satisfying (i) or (ii) are those with no root and no branching vertex, namely $V=\Z$ or $V=-\N$. In either case, reversing the direction of edges gives a tree isomorphic to $\Z$ or $\N$ respectively, and the corresponding $S_\lambda$ is isometrically conjugate to a weighted backward shift on that tree.
\end{proof}

In light of Propositions \ref{leafless prop} and \ref{forward shift prop}, when studying chain recurrence for shift operators on directed trees, our focus can be on weighted backward shifts on leafless trees. 

\section{Characterization of chain recurrent backward shifts on trees}\label{Characterization section} 

We now turn to our main results, which characterize chain recurrence of weighted backward shifts on the classical Banach sequence spaces of a directed tree purely in terms of the weights. By Propositions \ref{leafless prop} and \ref{forward shift prop}, it suffices to treat backward shifts on leafless trees. We distinguish two cases, according to whether the underlying tree is unrooted or rooted, and treat them in Theorem \ref{unrooted theorem} and Theorem \ref{rooted theorem}, respectively. In each case our characterization extends and refines the classical conditions known for weighted backward shifts on the integer-indexed sequence spaces $\ell^p(\Z)$, $\ell^p(\N)$, and $c_0(\Z)$, $c_0(\N)$ obtained in \cite[Proposition 20]{(1)Alves} and \cite[Theorem 14]{BernardesPeris}. The unrooted case is the more involved of the two, since chain recurrence demands the construction of $\delta$-chains in both directions through the tree: from $0_X$ to a basis vector $e_v$ (a ``forward'' construction descending into the children of $v$), and from $e_v$ back to $0_X$ (a ``backward'' construction ascending toward the unique infinite chain of ancestors of $v$). This bidirectional nature is reflected in our characterization by a pair of divergence conditions, one for each direction. In the rooted setting, the ancestor chain terminates at the root, so the backward construction becomes trivial; consequently, only the forward divergence condition survives. After establishing the two main theorems, we illustrate them with two natural classes of examples: Rolewicz operators (Example \ref{Rolewicz operators}) and symmetric weighted shifts on symmetric trees (Example \ref{symmetric trees example}). 

For $1<p<\infty$, let $p^*$ be the unique number in $(1,\infty)$ for which $1/p+1/p^*=1$. Since the unrooted case is more involved, we present it first.

\begin{theorem}\label{unrooted theorem}
    Let $(V,E)$ be an unrooted leafless directed tree and $\lambda=(\lambda_v)_{v\in V}$ a weight. We have the following:
    \begin{enumerate}[label=\upshape(\alph*)]
        \item Let $X=\ell^1(V)$, and let $B_\lambda$ be bounded on $X$. Then $B_\lambda$ is chain recurrent if and only if for every $v\in V$, both
        \begin{enumerate}[label=\upshape(\roman*)]
            \item $\sum_{n=0}^\infty\sup_{u\in\Child^n(v)}|\lambda(v\to u)|=\infty$, and
            \item $\lim_{n\to \infty} \sum_{j=0}^{n-1} \sup_{u\in \Child^j(\parent^n(v))} \left| \frac{\lambda (\parent^n(v)\to u)}{\lambda(\parent^n(v)\to v)} \right| = \infty$. 
        \end{enumerate}
        \item Let $X=\ell^p(V)$, $1<p<\infty$, and let $B_\lambda$ be bounded on $X$. Then $B_\lambda$ is chain recurrent if and only if for every $v\in V$, both
        \begin{enumerate}[label=\upshape(\roman*)]
            \item $\sum_{n=0}^\infty \big(\sum_{u\in \Child^n(v)} |\lambda(v\to u)|^{p^*}\big)^{1/p^*}=\infty$, and
            \item $\lim_{n\to \infty} \sum_{j=0}^{n-1} \left( \sum_{u\in \Child^{j} (\parent^n(v))} \left| \frac{\lambda (\parent^n(v)\to u)}{\lambda(\parent^n(v)\to v)} \right|^{p^*} \right)^{1/p^*} = \infty$. 
        \end{enumerate}
        \item Let $X=c_0(V)$, and let $B_\lambda$ be bounded on $X$. Then $B_\lambda$ is chain recurrent if and only if for every $v\in V$, both
         \begin{enumerate}[label=\upshape(\roman*)]
            \item $\sum_{n=0}^\infty \sum_{u\in \Child^n(v)} |\lambda(v\to u)|=\infty$, and
            \item $\lim_{n\to \infty}  \sum_{j=0}^{n-1} \sum_{u\in \Child^{j} (\parent^n(v))} \left| \frac{\lambda (\parent^n(v)\to u)}{\lambda(\parent^n(v)\to v)} \right|=\infty$. 
        \end{enumerate}
    \end{enumerate}
\end{theorem}

\begin{proof}
    We prove (b); the cases of $\ell^1(V)$ and $c_0(V)$ follow by entirely analogous arguments, with the H\"older duality step in the proof below replaced by the trivial estimate $\sup_u|\lambda(v\to u)| \cdot \|g_l\|_1$ for $\ell^1$ and by a direct sum bound for $c_0$. Before beginning, we first show that the sequence of sums that appear in condition (ii) are non-decreasing. For simplicity of notation, let $w_n = \parent^n(v)$, so that $w_{n+1} = \parent(w_n)$, and let $S_n$ denote the sum at the $n$th step:
\[
S_n = \sum_{j=0}^{n-1} \bigg(\sum_{u\in \Child^{j}(w_n)} \left| \frac{\lambda (w_n\to u)}{\lambda(w_n\to v)} \right|^{p^*}\bigg)^{1/p^*}.
\]
Consider $S_{n+1}$. Because $w_n \in \Child(w_{n+1})$, the descendants of $w_n$ form a subset of the descendants of $w_{n+1}$. Specifically, for $j \geq 1$, we have $\Child^{j-1}(w_n) \subseteq \Child^j(w_{n+1})$. Hence
\begin{align*}
S_{n+1} &= \sum_{j=0}^{n}\bigg( \sum_{u\in \Child^{j} (w_{n+1})} \left| \frac{\lambda (w_{n+1}\to u)}{\lambda(w_{n+1}\to v)} \right|^{p^*}\bigg)^{1/p^*} \\
&\geq \sum_{j=1}^{n}\bigg( \sum_{u\in \Child^{j-1}(w_n)} \left| \frac{\lambda (w_{n+1}\to u)}{\lambda(w_{n+1}\to v)} \right|^{p^*}\bigg)^{1/p^*}.
\end{align*}

For any $u \in \Child^{j-1}(w_n)$, the unique paths from both $u$ and $v$ to $w_{n+1}$ must pass through $w_n$. The path weights factor as $\lambda(w_{n+1} \to u) = \lambda(w_{n+1} \to w_n)\lambda(w_n \to u)$ and $\lambda(w_{n+1} \to v) = \lambda(w_{n+1} \to w_n)\lambda(w_n \to v)$. Hence the factor $\lambda(w_{n+1} \to w_n)$ cancels out:
\[
\left| \frac{\lambda (w_{n+1}\to u)}{\lambda(w_{n+1}\to v)} \right|^{p^*} = \left| \frac{\lambda (w_n\to u)}{\lambda(w_n\to v)} \right|^{p^*}.
\]

Substituting this identity into our inequality and re-indexing the sum by letting $k = j - 1$ yields:
\[
S_{n+1} \geq \sum_{k=0}^{n-1} \bigg(\sum_{u\in \Child^{k}(w_n)} \left| \frac{\lambda (w_n\to u)}{\lambda(w_n\to v)} \right|^{p^*}\bigg)^{1/p^*} = S_n,
\]
which shows $(S_n)$ is non-decreasing. Consequently, $\lim_{n\to \infty} S_n = \sup_{n} S_n$. We now prove (b).

    ($\Rightarrow$) Suppose $B_\lambda$ is chain recurrent.  Let $v\in V$ and $t>1$. Since $B_\lambda$ is chain recurrent, there must be a 1-chain from $te_v$ to itself of the form
\begin{align*}
    te_v=f_0,f_1,\ldots,f_{n}=te_v.
\end{align*}
By Lemma \ref{recurrence lemma}, there exist vectors $g_l\in \ell^p(V)$, $1\leq l\leq n$, with $\|g_l\|_p<1$ for every $l=1,2,\ldots,n$, satisfying 
\begin{align*}
    te_v=B_\lambda^n (te_v)+ \sum_{l=1}^n B_\lambda^{n-l}(g_l).
\end{align*}
Evaluating both sides of the previous expression at vertex $v$, using (\ref{B^n formula}) for the right-hand side, gives us
\begin{align*}
    t=0+\sum_{l=1}^n \sum_{u\in \Child^{n-l}(v)} \lambda(v\to u) g_l(u).
\end{align*}
By the triangle inequality and H\"older's inequality, applied to each inner sum, we estimate 
\begin{align*} 
    t & \leq \sum_{l=1}^n \bigg| \sum_{u\in \Child^{n-l}(v)} \lambda(v\to u) g_l(u)\bigg|\\
    & \leq \sum_{l=1}^{n}\bigg(\sum_{u\in\Child^{n-l}(v)}|\lambda(v\to u)|^{p^*}\bigg)^{1/p^*}\bigg(\sum_{u\in\Child^{n-l}(v)}|g_l(u)|^{p}\bigg)^{1/p}\\
    & \leq \sum_{l=1}^{n}\bigg(\sum_{u\in\Child^{n-l}(v)}|\lambda(v\to u)|^{p^*}\bigg)^{1/p^*} \| g_l\|_p \\
    & <  \sum_{l=1}^{n}\bigg(\sum_{u\in\Child^{n-l}(v)}|\lambda(v\to u)|^{p^*}\bigg)^{1/p^*} = \sum_{k=0}^{n-1}\bigg(\sum_{u\in\Child^{k}(v)}|\lambda(v\to u)|^{p^*}\bigg)^{1/p^*}.
\end{align*}
Since $t>1$ can be arbitrarily large, we must have $\sum_{n=0}^\infty \big(\sum_{u\in \Child^n(v)} |\lambda(v\to u)|^{p^*}\big)^{1/p^*}=\infty$, which verifies the first characterizing condition listed in the theorem.

The other condition is derived similarly. Since $B_\lambda$ is chain transitive, there must be a $1$-chain from $-te_v$ to $0_X$ of the form
\begin{align*}
    -te_v=p_0,p_1,\ldots,p_m=0.
\end{align*}
By Lemma \ref{recurrence lemma}, there exist vectors $q_j\in \ell^p(V)$, $1\leq j \leq m$, with $\|q_j\|_p<1$ for every $j=1,2,\ldots,m$, satisfying 
\begin{align*}
    0=B_\lambda^m (p_0)+\sum_{j=1}^m B_\lambda^{m-j}(q_j).
\end{align*}
Since $B_\lambda^m (p_0)=-t\lambda(\parent^m(v)\to v) e_{\parent^m(v)}$, evaluating both sides of the previous expression at the vertex $\parent^m(v)$ and rearranging yields
\begin{align*}
    t\lambda(\parent^m(v)\to v)=\sum_{j=1}^m \sum_{u\in \Child^{m-j}(\parent^m(v))} \lambda(\parent^m(v)\to u) q_j(u).
\end{align*}
Thus the triangle inequality yields
\begin{align*}
    t \leq \sum_{j=1}^m \bigg| \sum_{u\in \Child^{m-j}(\parent^m(v))} \frac{\lambda(\parent^m(v)\to u)}{\lambda(\parent^m(v)\to v)} q_j(u)\bigg|.
\end{align*}
Similar to earlier, H\"older's Inequality together with $\|q_j\|_p<1$ implies
\begin{align*}
    t & < \sum_{j=1}^m \bigg(\sum_{u\in \Child^{m-j}(\parent^m(v))} \left|\frac{\lambda(\parent^m(v)\to u)}{\lambda(\parent^m(v)\to v)} \right|^{p^*}\bigg)^{1/p^*}\\
    & =  \sum_{j=0}^{m-1}\bigg(\sum_{u\in \Child^{j}(\parent^m(v))} \left|\frac{\lambda(\parent^m(v)\to u)}{\lambda(\parent^m(v)\to v)} \right|^{p^*}\bigg)^{1/p^*}, 
\end{align*}
from which the second characterizing condition in the theorem follows, since $t>1$ can be arbitrarily large and the sequence of sums is increasing in $m$.

    ($\Leftarrow$)  Assume that (b-i) and (b-ii) hold.  By Lemma \ref{recurrence lemma}(b), it suffices to show that for every $\delta>0$ and every $v\in V$, there is a $\delta$-chain from $0_X$ to $e_v$ and a $\delta$-chain from $e_v$ to $0_X$. Fix $v \in V$ and $\delta > 0$. We first find a $\delta$-chain from $0_X$ to $e_v$ of the form
    \[0_X=f_0,f_1,\ldots,f_n=e_v.\]

    For each $n \geq 0$, let $A_n \coloneqq \Big( \sum_{u \in \Child^n(v)} |\lambda(v \to u)|^{p^*} \Big)^{1/p^*}$. By condition (b-i), the series $\sum_{n=0}^\infty A_n$ diverges. Therefore, we can choose an integer $n \geq 1$ large enough such that 
    \begin{align}\label{eq:t_sum}
        t \coloneqq \sum_{l=1}^n A_{n-l} = \sum_{k=0}^{n-1} A_k > \frac{1}{\delta}.
    \end{align}
    We now construct a sequence of vectors $g_l \in \ell^p(V)$ for $1 \leq l \leq n$ which will be used to create our $\delta$-chain. For each $l$, the vector $g_l$ will be supported entirely on $\Child^{n-l}(v)$. For  $u \in \Child^{n-l}(v)$, we define $g_l(u)$ using the exact equality condition for H\"older's inequality:
    \[
        g_l(u) \coloneqq \frac{1}{t} (A_{n-l})^{-p^*/p} \, |\lambda(v \to u)|^{p^*-1} \, \overline{\operatorname{sgn}(\lambda(v \to u))},
    \]
    where $\sgn(x)=x/|x|$ is the \emph{sign function} on $\K\setminus\{0\}$. We now check that this choice has the required properties. 
    
    We first check the norm of $g_l$. Since $(p^*-1)p = p^*$, we have $|g_l(u)|^p = \frac{1}{t^p} (A_{n-l})^{-p^*} |\lambda(v \to u)|^{p^*}$. Summing over $u \in \Child^{n-l}(v)$ yields
    \[
        \|g_l\|_p^p = \frac{1}{t^p} (A_{n-l})^{-p^*} \sum_{u \in \Child^{n-l}(v)} |\lambda(v \to u)|^{p^*} = \frac{1}{t^p} (A_{n-l})^{-p^*} (A_{n-l})^{p^*} = \frac{1}{t^p}.
    \]
    Thus, $\|g_l\|_p = \frac{1}{t} < \delta$ for every $1 \leq l \leq n$. 

    Now, define $f_0 \coloneqq 0_X$ and recursively define $f_l \coloneqq B_\lambda(f_{l-1}) + g_l$ for every $1\leq l \leq n$. By Lemma \ref{recurrence lemma}(a), $(f_l)_{l=0}^n$ is a valid $\delta$-chain, and $f_n = \sum_{l=1}^n B_\lambda^{n-l}(g_l)$. We show $f_n=e_v$. For $u \in \Child^{n-l}(v)$, we know $B_\lambda^{n-l}(e_u) = \lambda(v \to u) e_v$. Because $g_l$ is supported on $\Child^{n-l}(v)$, applying the operator yields
    \[
        B_\lambda^{n-l}(g_l) = \sum_{u \in \Child^{n-l}(v)} g_l(u) \lambda(v \to u) e_v.
    \]
    Substituting our explicit formula for $g_l(u)$ gives
    \begin{align*}
        B_\lambda^{n-l}(g_l) &= \frac{1}{t} (A_{n-l})^{-p^*/p} \sum_{u \in \Child^{n-l}(v)} |\lambda(v \to u)|^{p^*-1} \overline{\operatorname{sgn}(\lambda(v \to u))} \lambda(v \to u) e_v \\
        &= \frac{1}{t} (A_{n-l})^{1-p^*} \sum_{u \in \Child^{n-l}(v)} |\lambda(v \to u)|^{p^*} e_v \\
        &= \frac{1}{t} (A_{n-l})^{1-p^*} (A_{n-l})^{p^*} e_v \\
        &= \frac{A_{n-l}}{t} e_v.
    \end{align*}
    Summing over all $l$, we find
    \[
        f_n = \sum_{l=1}^n B_\lambda^{n-l}(g_l) = \sum_{l=1}^n \frac{A_{n-l}}{t} e_v = \frac{1}{t} \left( \sum_{l=1}^n A_{n-l} \right) e_v = \frac{t}{t} e_v = e_v.
    \]
    Thus $f_0, \ldots, f_n$ is a valid $\delta$-chain from $0_X$ to $e_v$.

    Our last task is to find a $\delta$-chain from $e_v$ to $0_X$ of the form $p_0, p_1, \ldots, p_m$.  For each integer $m \geq 1$ and $0 \leq k \leq m-1$, let 
    \[
        C_{m,k} \coloneqq \left( \sum_{u\in \Child^{k} (\parent^m(v))} \left| \frac{\lambda (\parent^m(v)\to u)}{\lambda(\parent^m(v)\to v)} \right|^{p^*} \right)^{1/p^*}. 
    \]
    Set $T_m \coloneqq \sum_{k=0}^{m-1} C_{m,k}$. By condition (b-ii), $\lim_{m\to \infty} T_m = +\infty$. Therefore, we can choose an integer $m \geq 1$ such that $T_m > \frac{1}{\delta}$. 

    For each $1 \leq j \leq m$, we define a vector $q_j \in \ell^p(V)$ supported entirely on $\Child^{m-j}(\parent^m(v))$ which will be used to construct a $\delta$-chain. Notice that as an index $j$ runs from $1$ to $m$, the generation index $m-j$ exactly covers the set $\{0, 1, \dots, m-1\}$. Let $w_u \coloneqq \frac{\lambda(\parent^m(v)\to u)}{\lambda(\parent^m(v)\to v)}$. For $u \in \Child^{m-j}(\parent^m(v))$, define
    \[
        q_j(u) \coloneqq -\frac{1}{T_m} (C_{m,m-j})^{-p^*/p} \, |w_u|^{p^*-1} \, \overline{\operatorname{sgn}(w_u)}.
    \]
    By the exact same arithmetic as above, summing $|q_j(u)|^p$ over $u \in \Child^{m-j}(\parent^m(v))$ yields $\|q_j\|_p = \frac{1}{T_m} < \delta$.

    Define $p_0 \coloneqq e_v$ and $p_j \coloneqq B_\lambda(p_{j-1}) + q_j$ for every $1\leq j \leq m$. By Lemma \ref{recurrence lemma}(a), this forms a $\delta$-chain where
    \[
        p_m = B_\lambda^m(e_v) + \sum_{j=1}^m B_\lambda^{m-j}(q_j).
    \]
    We know $B_\lambda^m(e_v) = \lambda(\parent^m(v) \to v) e_{\parent^m(v)}$. For the sum terms, since $q_j$ is supported on $\Child^{m-j}(\parent^m(v))$, we have
    \begin{align*}
        B_\lambda^{m-j}(q_j) &= \sum_{u \in \Child^{m-j}(\parent^m(v))} q_j(u) \lambda(\parent^m(v) \to u) e_{\parent^m(v)} \\
        &= \lambda(\parent^m(v) \to v) \sum_{u \in \Child^{m-j}(\parent^m(v))} q_j(u) w_u e_{\parent^m(v)}.
    \end{align*}
    Substituting our explicit formula for $q_j(u)$ yields
    \begin{align*}
        \sum_{u \in \Child^{m-j}(\parent^m(v))} q_j(u) w_u &= -\frac{1}{T_m} (C_{m,m-j})^{-p^*/p} \sum_{u \in \Child^{m-j}(\parent^m(v))} |w_u|^{p^*-1} \overline{\operatorname{sgn}(w_u)} w_u \\
        &= -\frac{1}{T_m} (C_{m,m-j})^{1-p^*} \sum_{u \in \Child^{m-j}(\parent^m(v))} |w_u|^{p^*} \\
        &= -\frac{1}{T_m} C_{m,m-j}.
    \end{align*}
    Thus, $B_\lambda^{m-j}(q_j) = -\frac{1}{T_m} C_{m,m-j} \lambda(\parent^m(v) \to v) e_{\parent^m(v)}$. Summing over $j$ gives
    \begin{align*}
        \sum_{j=1}^m B_\lambda^{m-j}(q_j) &= \left( \sum_{j=1}^m -\frac{1}{T_m} C_{m,m-j} \right) \lambda(\parent^m(v) \to v) e_{\parent^m(v)} \\
        &= -\frac{1}{T_m} \left( \sum_{k=0}^{m-1} C_{m,k} \right) \lambda(\parent^m(v) \to v) e_{\parent^m(v)} \\
        &= -\frac{T_m}{T_m} \lambda(\parent^m(v) \to v) e_{\parent^m(v)} \\
        &= -\lambda(\parent^m(v) \to v) e_{\parent^m(v)}.
    \end{align*}
    Therefore, $p_m = \lambda(\parent^m(v) \to v) e_{\parent^m(v)} - \lambda(\parent^m(v) \to v) e_{\parent^m(v)} = 0_X$, which shows $p_0,\ldots,p_m$ is a $\delta$-chain from $e_v$ to $0_X$, finishing the proof.
\end{proof}

The rooted case follows quickly from the work above.

\begin{theorem}\label{rooted theorem}
    Let $(V,E)$ be a rooted directed tree with no leaves and $\lambda$ a weight. We have the following:
    \begin{enumerate}[label=\upshape(\alph*)]
        \item Let $X=\ell^1(V)$, and let $B_\lambda$ be bounded on $X$. Then $B_\lambda$ is chain recurrent if and only if for every $v\in V$, $\sum_{n=0}^\infty\sup_{u\in\Child^n(v)}|\lambda(v\to u)|=\infty$.
        \item Let $X=\ell^p(V)$, $1<p<\infty$, and let $B_\lambda$ be bounded on $X$. Then $B_\lambda$ is chain recurrent if and only if for every $v\in V$, $\sum_{n=0}^\infty \big( \sum_{u\in \Child^n(v)} |\lambda(v\to u)|^{p^*}\big)^{1/p^*}=\infty$.
        \item Let $X=c_0(V)$, and let $B_\lambda$ be bounded on $X$. Then $B_\lambda$ is chain recurrent if and only if for every $v\in V$, $\sum_{n=0}^\infty \sum_{u\in \Child^n(v)} |\lambda(v\to u)|=\infty$.
    \end{enumerate}
\end{theorem}

    \begin{proof}
        We prove (b); the cases of $\ell^1(V)$ and $c_0(V)$ follow by entirely analogous arguments. Fix $v\in V$. 

        ($\Rightarrow$) The argument used to derive condition (b-i) of Theorem \ref{unrooted theorem} began with a $1$-chain from $te_v$ to itself, and the manipulations that followed depended only on the formula \eqref{B^n formula} and on the existence of children, both of which are available in any leafless tree, rooted or not. Repeating that argument verbatim shows that if $B_\lambda$ is chain recurrent, then $\sum_{n=0}^\infty \big( \sum_{u\in \Child^n(v)} |\lambda(v\to u)|^{p^*}\big)^{1/p^*}=\infty$.

        ($\Leftarrow$) Assume the divergence condition holds for every $v\in V$. By Lemma \ref{recurrence lemma}(b), it suffices to show that for every $v\in V$ and every $\delta>0$, there is a $\delta$-chain from $0_X$ to $e_v$ and a $\delta$-chain from $e_v$ to $0_X$.
        
        The construction in the ($\Leftarrow$) direction of Theorem \ref{unrooted theorem} that produced a $\delta$-chain from $0_X$ to $e_v$ required only condition (b-i) and used solely the descendant structure of $v$. It therefore applies unchanged to the rooted setting and yields a $\delta$-chain from $0_X$ to $e_v$.

        A chain from $e_v$ to $0_X$ is easily obtained by the fact that $B_\lambda^n(e_v)=0_X$ for sufficiently large $n$. For such an $n$, setting $p_j\coloneqq B_\lambda^j(e_v)$ for $0\leq j\leq n$ yields $p_0=e_v$, $p_n=0_X$, and $p_{j+1}=B_\lambda(p_j)$ exactly. Thus $\|p_{j+1}-B_\lambda(p_j)\|=0$ for $0\leq j\leq n-1$, which shows $(p_j)_{j=0}^n$ is a $\delta$-chain for any $\delta>0$.
    \end{proof}

\begin{example}\label{Rolewicz operators}
    The simplest weighted shift to consider is when all the weights are the same, which is to consider a scalar multiple of the unweighted backward shift $B$. Let $\lambda\in \K$ be non-zero. Then the weight sequence defined by $\lambda_v=\lambda$ for each $v\in V$ corresponds to the \emph{Rolewicz operator} $\lambda B$, named after the important work of Rolewicz in \cite{Rolewicz}. As shown in \cite[Example 2.5]{GrossePapahypercyclic}, any Rolewicz operator is bounded on $\ell^1(V)$,  while it is bounded on $\ell^p(V)$, $1<p<\infty$ or on $c_0(V)$ if and only if $\sup_{v\in V} |\Child(v)|<\infty$. The conditions in Theorem \ref{unrooted theorem} can be simplified when considering a Rolewicz operator with the following result.

    \begin{corollary}\label{rolewicz corollary}
     Let $(V,E)$ be an unrooted leafless directed tree. Let $\lambda\in \K$ be non-zero and $\lambda B$ the corresponding Rolewicz operator. We have the following:
     \begin{enumerate}[label=\upshape(\alph*)]
         \item The operator $\lambda B$ is chain recurrent on $\ell^1(V)$ if and only if $|\lambda|=1$.
         \item Let $X=\ell^p(V), 1<p<\infty$, or $X=c_0(V)$, and let $\lambda B$ be bounded on $X$. Then $\lambda B$ is chain recurrent on $X$ if and only if for every $v\in V$, both
         \begin{enumerate}[label=\upshape(\roman*)]
             \item $\sum_{n=0}^\infty |\Child^n(v)|^{1/p^*} |\lambda|^n = \infty$, and
             \item $\lim_{n\to \infty} \sum_{j=0}^{n-1}|\Child^j(\parent^n(v))|^{1/p^*}|\lambda|^{(j-n)}=\infty$.
         \end{enumerate}
         If, additionally, the tree has a free left end, then condition (ii) reduces to the requirement $|\lambda|\leq 1$. Here, $p^*=1$ for $X=c_0(V)$. 
     \end{enumerate}
\end{corollary}

\begin{proof}
    Simplifying the conditions in Theorem \ref{unrooted theorem}(a) shows that $\lambda B$ is chain recurrent on $\ell^1(V)$ if and only if $\sum_{n=1}^\infty |\lambda|^n = \infty$ and $\lim_{n\to \infty}\sum_{j=0}^{n-1} |\lambda|^{j-n} = \infty$, which happens if and only if $|\lambda|=1$.

    To prove (b), let $X=\ell^p(V), 1<p<\infty$, or $X=c_0(V)$. The simplifications throughout the proof rest on the following two identities, which we shall use repeatedly in what follows: for every $v\in V$, $n\in \N$, and $0\leq j\leq n-1$,
    \begin{align*}
        \sum_{u\in \Child^n(v)} |\lambda(v\to u)|^{p^*} &= |\Child^n(v)||\lambda|^{np^*}, \text{ and} \\
        \sum_{u\in \Child^j(\parent^n(v))} \left| \frac{\lambda (\parent^n(v)\to u)}{\lambda(\parent^n(v)\to v)} \right|^{p^*} &= |\Child^j(\parent^n(v))||\lambda|^{(j-n)p^*}.
    \end{align*}
    Thus after simplifying the conditions in Theorem \ref{unrooted theorem} (parts (b) and (c)), $\lambda B$ is chain recurrent if and only if for every $v\in V$,
    \begin{enumerate}[label=\upshape(\roman*)]
        \item $\infty  = \sum_{n=0}^\infty (|\Child^n(v)||\lambda|^{np^*})^{1/p^*} = \sum_{n=0}^\infty |\Child^n(v)|^{1/p^*}|\lambda|^n$, and 
        \item $\infty = \lim_{n\to \infty} \sum_{j=0}^{n-1} (|\Child^j(\parent^n(v))||\lambda|^{(j-n)p^*})^{1/p^*} = \lim_{n\to \infty} \sum_{j=0}^{n-1} |\Child^j(\parent^n(v))|^{1/p^*}|\lambda|^{(j-n)}$.
    \end{enumerate} 
    
    Now suppose the tree has a free left end. Then there are vertices $v$ for which $\Child^j(\parent^n(v))=\{\parent^{n-j}(v)\}$ for every $0\leq j \leq n-1$. For such a vertex $v$, we have $\sum_{j=0}^{n-1} |\Child^j(\parent^n(v))|^{1/p^*} |\lambda|^{(j-n)}=\sum_{j=0}^{n-1} |\lambda|^{(j-n)}$, which approaches $\infty$ if and only if $|\lambda|\leq 1$. Hence $|\lambda|\leq 1$ is necessary. Conversely, if $|\lambda|\leq 1$, then since $|\Child^j(\parent^n(v))|\geq 1$, we have
    \begin{align*}
        \sum_{j=0}^{n-1} |\Child^j(\parent^n(v))|^{1/p^*} |\lambda|^{(j-n)} \geq \sum_{j=0}^{n-1} |\lambda|^{(j-n)} \to \infty,
    \end{align*}
    as desired, finishing the proof that (ii) is equivalent to $|\lambda|\leq 1$ when the tree has a free left end.
\end{proof}
    The rooted case involves the same calculations as the unrooted case.
    \begin{corollary}
         Let $(V,E)$ be a rooted directed leafless tree. Let $\lambda\in \K$ be non-zero and $\lambda B$ the corresponding Rolewicz operator. We have the following:
         \begin{enumerate}[label=\upshape(\alph*)]
             \item The operator $\lambda B$ is chain recurrent on $\ell^1(V)$ if and only if $|\lambda|\geq 1$.
             \item Let $X=\ell^p(V), 1<p<\infty$, or $X=c_0(V)$, and let $\lambda B$ be bounded on $X$. Then $\lambda B$ is chain recurrent on $X$ if and only if $\sum_{n=0}^\infty |\Child^n(v)|^{1/p^*} |\lambda|^{n} = \infty$.

             Here, $p^*=1$ for $X=c_0(V)$.
             \end{enumerate}
    \end{corollary}

    \begin{proof}
        Simplifying the condition in Theorem \ref{rooted theorem}(a) shows $\lambda B$ is chain recurrent on $\ell^1(V)$ if and only if $\sum_{n=1}^\infty |\lambda|^n=\infty$, which happens if and only if $|\lambda|\geq 1$. For (b), let $X=\ell^p(V), 1<p<\infty$, or $X=c_0(V)$. Simplifying the other conditions in Theorem \ref{rooted theorem} shows $\lambda B$ is chain recurrent on $X$ if and only if  $\sum_{n=1}^\infty |\Child^n(v)|^{1/p^*} |\lambda|^{n} = \infty$ for every $v\in V$. Since the $n=0$ term is $|\Child^0(v)|^{1/p^*} |\lambda|^{0} = 1$, this series diverges if and only if $\sum_{n=0}^\infty |\Child^n(v)|^{1/p^*} |\lambda|^{n} = \infty$.
    \end{proof}
\end{example}

\begin{example}\label{symmetric trees example}
    Let $(V,E)$ be a directed tree and fix $v_0$ so as to enumerate the generations of the tree, picking $v_0=\mathsf{root}$ in the rooted case and picking an arbitrary vertex $v_0\in V$ otherwise. A weight $\lambda=(\lambda_v)_{v\in V}$ is called \emph{symmetric} if any two vertices belonging to the same generation have identical corresponding weights. For $v\in \Gen_n$, we can then write $\lambda_n=\lambda_v$. We say that a directed tree $(V,E)$ is \emph{symmetric} if any two vertices belonging to the same generation have the same number of children. For $v\in \Gen_n$ of a symmetric tree, we set $\gamma_n = |\Child^n(v)|$. The class of symmetric weighted backward shifts on symmetric trees form another notable class of examples, since they generalize the classical weighted backward shifts on $\N$ and $\Z$. We first record a simple computational lemma. 

    \begin{lemma}\label{simplify sum lemma}
        Let $(a_n)_{n\in \Z}$ be a (two-tailed) sequence of positive terms. Then the following are equivalent.
        \begin{enumerate}[label=\upshape(\roman*)]
            \item $\ds\sum_{n=0}^\infty \prod_{i=0}^n a_i =\infty$;
            \item for every $m\in \Z$, $\ds \sum_{n=0}^\infty \prod_{i=m}^{m+n} a_i =\infty$.
        \end{enumerate}
        Furthermore, the following are equivalent as well.
        \begin{enumerate}[label=\upshape(\roman*)]
        \setcounter{enumi}{2}
            \item $\ds\sum_{n=0}^\infty \prod_{i=0}^{n-1} a_{-i} =\infty$;
            \item for every $m\in \Z$, $\ds \sum_{n=0}^\infty \prod_{i=m-n}^{m-1} a_{-i} =\infty$.
        \end{enumerate}
    \end{lemma}

    \begin{proof}
    We only prove (i)$\Leftrightarrow$(ii) since the proof of (iii)$\Leftrightarrow$(iv) is essentially the same. For an integer $k\geq \max(1, -m)$, the $k$th partial sums of the series in (i) and (ii) satisfy
    \begin{align*}
        \text{for } m>0,& \sum_{n=0}^k \prod_{i=m}^{m+n} a_i =\frac{1}{a_0\cdots a_{m-1}} \sum_{n=0}^{m+k} \prod_{i=0}^n a_i - \frac{1}{a_0\cdots a_{m-1}} \sum_{n=0}^{m-1} \prod_{i=0}^n a_i, \text{ and}\\
        \text{for } m<0,& \sum_{n=0}^k \prod_{i=m}^{m+n} a_i = \sum_{n=0}^{-m-1}\prod_{i=m}^{m+n}a_i + \bigg(\prod_{i=m}^{-1} a_i\bigg)\sum_{n=0}^{m+k}\prod_{i=0}^n a_i,
    \end{align*}
    from which the equivalence of (i) and (ii) can be seen.
\end{proof}
    
    We use this lemma to show how the conditions of Theorems \ref{unrooted theorem} and \ref{rooted theorem} simplify in the case of a symmetric backward shift on a symmetric tree.

    \begin{corollary}\label{unrooted symmetric corollary}
         Let $(V,E)$ be an unrooted leafless symmetric directed tree and $\lambda=(\lambda_n)_{n\in \Z}$ be a symmetric weight on $V$. We have the following:
         \begin{enumerate}[label=\upshape(\alph*)]
             \item A continuous symmetric shift $B_\lambda$ is chain recurrent on $\ell^1(V)$ if and only if both
             \begin{enumerate}[label=\upshape(\roman*)]
                 \item $\sum_{n=0}^\infty \prod_{i=1}^n |\lambda_i|=\infty$, and
                 \item $\sum_{n=0}^\infty \prod_{i=0}^{n-1} |\lambda_{-i}|^{-1}=\infty$.
             \end{enumerate}
             \item Let $X=\ell^p(V), 1<p<\infty$, or $X=c_0(V)$, and let $B_\lambda$ be bounded on $X$. Then $B_\lambda$ is chain recurrent on $X$ if and only if 
             \begin{enumerate}[label=\upshape(\roman*)]
                 \item $\ds\sum_{n=0}^\infty \prod_{i=1}^n (\gamma_{i-1})^{1/p^*}|\lambda_i|=\infty$, and
                 \item for every $m\in \Z$, 
                 \[\ds\lim_{n\to \infty} \left(\frac{1}{|\lambda_{m-n+1}\cdots \lambda_m|}+\sum_{j=1}^{n-1} \frac{(\gamma_{m-n}\cdots \gamma_{m-n+j-1})^{1/p^*}}{|\lambda_{m-n+j+1}\cdots \lambda_m|}\right) = \infty.\] 
             \end{enumerate}
             If, additionally, the tree has a free left end, then condition (ii) reduces to 
             \begin{enumerate}[label=\upshape(\roman*)]
             \setcounter{enumii}{2}
                 \item $\ds\sum_{n=0}^\infty \prod_{i=0}^n \frac{1}{|\lambda_{-i}|}=\infty$.
             \end{enumerate}
             Here, $p^*=1$ for $X=c_0(V)$.
         \end{enumerate}
    \end{corollary}

    \begin{proof}
    As usual, we fix some $v_0\in V$ to enumerate the generations of the tree. In the case of $\ell^1(V)$ in part (a), after simplifying the two conditions in Theorem \ref{unrooted theorem}(a), we see that $B_\lambda$ is chain recurrent if and only if for every $m\in \Z$,
     \begin{enumerate}
        \item[(i')] $\sum_{n=1}^\infty\prod_{i=m+1}^{m+n}|\lambda_i|=\infty$, and
        \item[(ii')] $\lim_{n\to \infty} \sum_{j=0}^{n-1} \prod_{i=m-n+j+1}^m \frac{1}{|\lambda_i|} = \infty$.
    \end{enumerate}
    Since $\lim_{n\to\infty}\sum_{j=0}^{n-1}\prod_{i=m-n+j+1}^{m}\frac{1}{|\lambda_i|} = \frac{1}{|\lambda_m|}\sum_{n=0}^{\infty}\prod_{i=m-n}^{m-1}\frac{1}{|\lambda_i|}$, Lemma \ref{simplify sum lemma} shows conditions (i') and (ii') are equivalent to (i) and (ii) in the statement of Corollary \ref{unrooted symmetric corollary}(a).

    For case (b), define the sequence $(a_k)_{k\in\Z}$ by $a_k=(\gamma_{k-1})^{1/p^*}|\lambda_k|$. Let $v\in V$ and suppose $v\in \Gen_m$, where $m\in \Z$. By the symmetry of the tree, we have  $|\Child^n(v)|=\gamma_m\cdots\gamma_{m+n-1}$ for $n\geq 1$. Now the first condition of Theorem \ref{unrooted theorem}(b) (or equivalently of (c) when $X=c_0(V)$) is equivalent to
    \begin{align*}
        \infty=\sum_{n=1}^\infty\bigg(\sum_{u\in \Child^n(v)} |\lambda(v\to u)|^{p^*}\bigg)^{1/p^*} = \sum_{n=1}^\infty \bigg(\prod_{i=m+1}^{m+n} \gamma_{i-1}|\lambda_i|^{p^*}\bigg)^{1/p^*}=\sum_{n=1}^\infty \prod_{i=1}^n a_{m+i},
    \end{align*}
    for each $m\in \Z$. By Lemma \ref{simplify sum lemma}, the above condition is equivalent to $\sum_{n=1}^\infty \prod_{i=1}^n a_i = \infty$. Since $a_0 > 0$, this diverges if and only if $1 + \sum_{n=1}^\infty \prod_{i=1}^n a_i = \infty$, which evaluates exactly to $\sum_{n=0}^\infty \prod_{i=1}^n (\gamma_{i-1})^{1/p^*}|\lambda_i|=\infty$, yielding the first condition.

    We now simplify the second condition from Theorem \ref{unrooted theorem}. Again let $v\in V$ and let $m\in \Z$ be the unique integer for which $v\in \Gen_m$. Then $\parent^n(v) \in \Gen_{m-n}$. For a fixed $j \in \{0, \ldots, n-1\}$, the number of children at generation $j$ of the ancestor $\parent^n(v)$ is $|\Child^j(\parent^n(v))| = \prod_{i=m-n}^{m-n+j-1} \gamma_i$. For any $u\in \Child^j(\parent^n(v))$, the ratio of weights is:
    \begin{align*}
        \left| \frac{\lambda(\parent^n(v) \to u)}{\lambda(\parent^n(v) \to v)} \right|^{p^*} = \frac{\prod_{i=m-n+1}^{m-n+j} |\lambda_i|^{p^*}}{\prod_{i=m-n+1}^{m} |\lambda_i|^{p^*}} = \frac{1}{\prod_{i=m-n+j+1}^m |\lambda_i|^{p^*}}.
    \end{align*} 
    Taking the $1/p^*$ power of the inner sums $T_{j,n}\coloneqq \sum_{u\in \Child^j(\parent^n(v))} \left| \frac{\lambda(\parent^n(v) \to u)}{\lambda(\parent^n(v) \to v)} \right|^{p^*}$ yields
    \begin{align*}
        (T_{j,n})^{1/p^*} = \frac{\prod_{i=m-n}^{m-n+j-1} \gamma_i^{1/p^*}}{\prod_{i=m-n+j+1}^m |\lambda_i|}.
    \end{align*}
    Defining $\Gamma_{n,m} = \prod_{i=m-n}^{m-1} \gamma_i^{1/p^*}$, we can rewrite the expression for $(T_{j,n})^{1/p^*}$ by multiplying the numerator and denominator by $\prod_{i=m-n+j}^{m-1} \gamma_i^{1/p^*}$ to get:
    \begin{align*}
        (T_{j,n})^{1/p^*} = \frac{\Gamma_{n,m}}{\left(\prod_{i=m-n+j}^{m-1} \gamma_i^{1/p^*} \right) \left( \prod_{i=m-n+j+1}^m |\lambda_i| \right)} = \frac{\Gamma_{n,m}}{\prod_{i=m-n+j+1}^m \gamma_{i-1}^{1/p^*}|\lambda_i|} = \frac{\Gamma_{n,m}}{\prod_{i=m-n+j+1}^m a_i}.
    \end{align*}
    The full sum for condition (ii) is:
    \begin{align*}
        S_{n,m} = \sum_{j=0}^{n-1} (T_{j,n})^{1/p^*} = \Gamma_{n,m} \sum_{j=0}^{n-1} \bigg(\prod_{i=m-n+j+1}^m a_i\bigg)^{-1}.
    \end{align*}
    By changing the index $k = n-j$, the sum becomes:
    \begin{align*}
        S_{n,m} = \Gamma_{n,m} \sum_{k=1}^n \bigg(\prod_{i=m-k+1}^m a_i\bigg)^{-1}.
    \end{align*}
    Now $S_{n,m}$ is exactly the expression in the limit of Corollary \ref{unrooted symmetric corollary}(b-ii). 

    If the tree has a free left end, then for all sufficiently negative $i$, we have $\gamma_i = 1$. Consequently, $\Gamma_{n,m}$ is bounded above by some constant (dependent on $m$) and bounded below by $1$. In this case, $S_{n,m}$ diverges for each $m\in \Z$ if and only if $\sum_{n=1}^\infty \frac{1}{\prod_{i=m-n+1}^m a_i} = \infty$ for each $m\in \Z$. By letting $c_i = 1/a_i$ and re-indexing $l = m-i$, this sum can be written as $\sum_{n=1}^\infty \prod_{l=0}^{n-1} c_{m-l} = \infty$, which by Lemma \ref{simplify sum lemma} happens if and only if $\sum_{n=1}^\infty \prod_{l=0}^{n-1} c_{-l} = \infty$. Because the tree has a free left end, there exists some integer $M \geq 0$ such that $\gamma_{-l-1}=1$ for all $l \geq M$. Thus for $l \geq M$, $c_{-l} = \frac{1}{\gamma_{-l-1}^{1/p^*}|\lambda_{-l}|} = \frac{1}{|\lambda_{-l}|}$. The terms of the series $\sum_{n=1}^\infty \prod_{l=0}^{n-1} c_{-l}$ thus differ from the terms of $\sum_{k=0}^\infty \prod_{l=0}^k \frac{1}{|\lambda_{-l}|}$ by at most a finite number of constant factors $\gamma_{-l-1}^{-1/p^*}$ at the beginning of the products. Since these factors are strictly positive, the series diverge together, which yields condition (iii).
\end{proof}

    As the proof of the unrooted case shows, we have the following in the rooted case.

   \begin{corollary}
         Let $(V,E)$ be a rooted leafless symmetric directed tree and $\lambda=(\lambda_n)_{n\in \N_0}$ be a symmetric weight on $V$. We have the following:
         \begin{enumerate}[label=\upshape(\alph*)]
             \item A continuous symmetric shift $B_\lambda$ is chain recurrent on $\ell^1(V)$ if and only if $\sum_{n=1}^\infty \prod_{i=1}^n |\lambda_i|=\infty$.
             \item Let $X=\ell^p(V), 1<p<\infty$, or $X=c_0(V)$, and let $B_\lambda$ be bounded on $X$. Then $B_\lambda$ is chain recurrent on $X$ if and only if $\ds\sum_{n=1}^\infty \prod_{i=1}^n (\gamma_{i-1})^{1/p^*}|\lambda_i|=\infty$.
             
             Here, $p^*=1$ for $X=c_0(V)$.
         \end{enumerate}
    \end{corollary}
\end{example}

We close by noting that the two symmetric corollaries above specialize, in their simplest case, to the classical sequence-space characterizations of chain recurrence on the Banach spaces $\ell^p(\N)$, $\ell^p(\Z)$ ($1\leq p<\infty$), $c_0(\N)$, and $c_0(\Z)$. Viewing $\Z$ (resp.\ $\N$) as the unrooted (resp.\ rooted) leafless symmetric directed tree in which every vertex has a unique child, every generation is a singleton, so $\gamma_n=1$ for all $n$, and in the unrooted case the tree has a free left end. A symmetric weight thereby reduces to an ordinary scalar sequence, and our corollaries collapse to the following.

\begin{corollary}\label{classical corollary}
    Let $w=(w_n)$ be a sequence of nonzero scalars and let $B_w$ be the corresponding weighted backward shift, acting on basis vectors by $B_w(e_n)=w_n e_{n-1}$.
    \begin{enumerate}[label=\upshape(\alph*)]
        \item Let $X=\ell^p(\Z)$, $1\leq p<\infty$, or $X=c_0(\Z)$, and let $B_w$ be bounded on $X$. Then $B_w$ is chain recurrent on $X$ if and only if
        \[\sum_{n=1}^\infty |w_1 w_2\cdots w_n|=\infty \quad \text{and} \quad \sum_{n=1}^\infty\frac{1}{|w_{-n+1}\cdots w_{-1} w_0|}=\infty.\]
        \item Let $X=\ell^p(\N)$, $1\leq p<\infty$, or $X=c_0(\N)$, and let $B_w$ be bounded on $X$. Then $B_w$ is chain recurrent on $X$ if and only if $\ds\sum_{n=1}^\infty |w_1 w_2\cdots w_n|=\infty$.
    \end{enumerate}
\end{corollary}

\begin{proof}
    Identify $w$ with the symmetric weight on the chain tree as above. Substituting $\gamma_n=1$ into Corollary~\ref{unrooted symmetric corollary}(a)(i)--(ii) (for $\ell^1(\Z)$) and into Corollary~\ref{unrooted symmetric corollary}(b)(i),~(iii) (for $\ell^p(\Z)$, $1<p<\infty$, and $c_0(\Z)$), and dropping the $n=0$ summand (which contributes the finite constant~$1$) from each series, gives the two divergence conditions in (a). For (b), substituting $\gamma_n=1$ into the rooted symmetric corollary above gives the single condition in the form stated.
\end{proof}

This is precisely the characterization of chain recurrence on the classical Banach sequence spaces obtained by Bernardes and Peris in \cite[Theorems 14 and 16]{BernardesPeris}, where the more general setting of Fr\'echet sequence spaces is considered; the $\ell^1$ case had earlier been established in \cite[Proposition 20]{(1)Alves}. 
\bibliography{Bib}

\end{document}